\documentclass[12pt]{amsart}
\usepackage{amscd,amsrefs}
\newtheorem{theorem}{Theorem}[section]
 \newtheorem{corollary}[theorem]{Corollary}
 \newtheorem{lemma}[theorem]{Lemma}
 \newtheorem{proposition}[theorem]{Proposition}

 \newtheorem{conjecture}[theorem]{Conjecture}
 \theoremstyle{definition}
 \newtheorem{definition}[theorem]{Definition}
 \theoremstyle{remark}
 \newtheorem{remark}[theorem]{Remark}
 \newtheorem{notation}[theorem]{Notation}
 \numberwithin{equation}{section}

\begin{document}

\title{Homological symbols and the Quillen Conjecture}%
\author{Marian F. Anton}%
\address{Department of Mathematics, University of Kentucky, Lexington, KY 40506-0027, U.S.A. and I.M.A.R., P.O. Box 1-764, Bucharest, RO 014700}%
\email{anton@ms.uky.edu}%


\begin{abstract}
We formulate a "correct" version of the Quillen conjecture on
linear group homology for certain arithmetic rings and provide
evidence for the new conjecture. In this way we predict that the
linear group homology has a direct summand looking like an
unstable form of Milnor K-theory and we call this new theory
"homological symbols algebra". As a byproduct we prove the Quillen
conjecture in homological degree two for the rank two and the
prime $5$.
\end{abstract}
\maketitle

\section{Introduction}

Let $R$ be a subring with identity of the complex numbers $\mathbb
C$ and resp. $GL_n$, $SL_n$ the discrete group of $n\times n$
matrices over $R$ with determinant resp. nonzero, $1$. If
$H(GL_n):=H^*(GL_n;\mathbb F_\ell)$ denotes the mod $\ell$ group
cohomology of $GL_n$, then the canonical inclusion
$R\subset\mathbb C$ induces a module structure of $H(GL_n)$ over
the singular mod $\ell$ cohomology ring of Chern classes
$P_n:=H^*(BGL_n(\mathbb C);\mathbb F_\ell)$ where $BGL_n(\mathbb
C)$ denotes the classifying space of the Lie group $GL_n(\mathbb
C)$ of invertible $n\times n$ matrices over $\mathbb C$. In
\cite{MR0298694}*{p. 591} Quillen conjectured that for certain
primes $\ell$ and rings $R$ the module $H(GL_n)$ is \emph{free}
over $P_n$. We call this statement \emph{the \emph{strong} Quillen
conjecture for the rank $n$ and the prime $\ell$}.

In particular, if we fix $R=\mathbb Z[\frac{1}{\ell},\xi_\ell]$
where $\ell$ is a \emph{regular} prime and $\xi_\ell\in\mathbb C$
is a primitive $\ell$-th root of unity, then it has been shown in
\cite{MR1312569}*{p. 51} that the strong Quillen's conjecture
implies that \emph{the homomorphism
$$\iota_{np}:H_p(GL_1^{\times n};\mathbb F_\ell)\to H_p(GL_n;\mathbb
F_\ell)$$ induced by the canonical inclusion $GL_1^{\times
n}\subset GL_n$ on mod $\ell$ homology is \emph{surjective} for
all $p$}. We call the statement that $\iota_{np}$ is surjective
\emph{the \emph{weak} Quillen conjecture in homological degree $p$
for the rank $n$ and the prime $\ell$}. This weak conjecture was
disproved in \cite{MR1443381} for $n\ge 32$, $\ell=2$, and in
\cite{MR1723188} for $n\ge 27$, $\ell=3$, in the sense that there
is an unspecified $p$ depending on $n$ and $\ell$ for which the
statement fails.

In this article we formulate yet another conjecture for $\ell$
\emph{odd} and regular (Conjecture \ref{cj}) which proves that the
weak Quillen conjecture for the rank $n$, the prime $\ell$ and
\emph{all} homological degrees $p$ implies the strong Quillen
conjecture for the same rank $n$ and prime $\ell$ (see subsection
\S \ref{statement} for a full discussion). More specifically, by
Proposition \ref{reduce}, this new conjecture states  that a
certain \emph{finite} set of homological classes in
$H_*(GL_1;\mathbb F_\ell)$ vanish in $H_*(SL_2;\mathbb F_\ell)$
under the map induced from embedding $GL_1$ in $SL_2$ via
$u\mapsto\begin{pmatrix}u^{-1}&0\\0&u\end{pmatrix}$. These classes
are called \emph{\'{e}tale obstruction classes} since they
originate from studying \'{e}tale models \cite{MR1973992} for the
classifying spaces $BGL_n$. The bar complex cycles representing
these classes are given explicitly in Definition \ref{eos}.

As evidence for the Conjecture \ref{cj} we remark that this
conjecture and the weak Quillen conjecture for all $p$ are true
for $\ell=3$ by direct calculations \cite{MR1723188} and thus, the
strong Quillen conjecture holds in this case. Also the case
$\ell=2$ fits into the same pattern for the ranks $n=2$
\cite{MR1147814} and $n=3$ \cite{MR1683179}. In this article we
prove a new result stating that
\begin{theorem}\label{main1}
$H_2(SL_2(\mathbb Z[\frac{1}{5},\xi_5];\mathbb F_5)=0$.
\end{theorem}
As a corollary, our conjecture is true in homological degree two
for $\ell=5$ in the sense that the \'{e}tale obstruction classes
from $H_2(GL_1;\mathbb F_5)$ obviously vanish in $H_2(SL_2;\mathbb
F_5)$. As a byproduct, we obtain that the weak Quillen conjecture
in homological degree two for the rank two and the prime $5$ is
also true and

\begin{theorem}\label{main2}
$H_2(GL_2(\mathbb Z[\frac{1}{5},\xi_5];\mathbb F_5)\approx \mathbb
F_5\oplus\mathbb F_5\oplus\mathbb F_5\oplus\mathbb F_5$.
\end{theorem}

The technique used in proving Theorem \ref{main1} is based on
proving a key result in Theorem \ref{main} regarding the structure
of the group $SL_2$ as a finitely presented group and using GAP
\cite{gap} in a clever way. The main difficulties reside in the
complexity of the combinatorial group problems associated with
Hopf's formula and its generalizations \cite{MR997360}.

Another feature of this article is a characterization of a direct
summand (as a vector space) of the bigraded algebra
\begin{equation}\label{a}
A:=\bigoplus_{i,j=0}^\infty H_i(GL_j;\mathbb F_\ell)
\end{equation}
where the algebra structure is induced from the matrix block
multiplication. This summand is the bigraded subalgebra $KA\subset
A$ generated by the linear subspace $H_*(GL_1;\mathbb
F_\ell)\subset A$ and its structure is predicted by the new
conjecture in the sense that the relations in $KA$ come from
$H_*(GL_1^{\times 2};\mathbb F_\ell)$ in a certain explicit way
(see Remark \ref{useful}).

We recall \cite{MR0260844} that the (naive) Milnor $K$-theory of
the ring $R$ is the tensor algebra generated by the group of units
$GL_1$ modulo the Steinberg relations $u\otimes(1-u)=0$ coming
from $GL_1^{\otimes 2}$ for $u,1-u\in GL_1$. By replacing $GL_1$
with $H_*(GL_1;\mathbb F_\ell)$, $GL_1^{\otimes2}$ with
$H_*(GL_1^{\times 2};\mathbb F_\ell)$ and the Steinberg relations
with those relations predicted by our conjecture, we obtain the
conjectural structure of $KA$. For this reason, we call $KA$ the
\emph{algebra of homological symbols at $\ell$} associated with
the ring $R$.

The paper is organized as follows. After reviewing some basic
group homology facts in \S \ref{prelim} and introducing some
algebra terminology in \S \ref{algebra}, we describe the direct
summand of the algebra \eqref{a} and estimate from "above" the
relations of this summand in Theorem \ref{kernel}. The conjecture
on the exact relations is formulated in \S \ref{conj}. In \S
\ref{fp} we estimate the relations in $SL_2$ from "below" for
\emph{any} regular odd prime and use them in \S \ref{hopff} to
prove Theorem \ref{main1} (see Corollary \ref{proof}). Theorem
\ref{main2} follows now from Theorems \ref{main1} and \ref{kernel}
by a spectral sequence argument.

\section{Group homology preliminaries}\label{prelim}

We recall some standard facts about group homology as in
\cite{Brown}. Let $G$ be a multiplicative group with neutral
element $1\in G$ and $k$ a commutative ring with identity.

\subsection{The shuffle product}
Let $\mathcal B_*(G;k)$ be the normalized bar complex:
\begin{equation}\label{chain}
\mathcal B_0(G;k)\xleftarrow{\partial}\mathcal B_1(G;k)...\xleftarrow{\partial}\mathcal
B_{s-1}(G;k)\xleftarrow{\partial}\mathcal B_s(G;k)\xleftarrow{\partial}...
\end{equation}
where $\mathcal B_s(G;k)$ is the free $k$-module generated by the set of symbols $[x_1|...|x_s]$ with $x_1,...,x_s\in
G\setminus\{1\}$ and $\partial$ is the $k$-homomorphism given by the formula:
\begin{equation*}
\partial[x_1|...|x_s]=[x_2|...|x_s]+\sum_{j=1}^{s-1}(-1)^j[x_1|...|x_jx_{j+1}|...|x_s]+(-1)^s[x_1|...|x_{s-1}]
\end{equation*}
with $[x_1|...|x_jx_{j+1}|...|x_s]=0$ by convention if $x_jx_{j+1}=1$. By definition, the group homology $H_*(G;k)$
with $k$-coefficients is the homology of the chain complex \eqref{chain}.

On the other hand, the chain complex \eqref{chain} can be regarded as a graded algebra $\mathcal B(G;k)$ over $k$ which
is anti-commutative, associative, and unital with respect to the shuffle product
\begin{equation}\label{shuffle}
[x_1|...|x_i]\wedge[x_{i+1}|...|x_{i+s}]=\sum(-1)^\sigma[x_{\sigma(1)}|...|x_{\sigma(i+s)}]
\end{equation}
where the sum is over all the permutations $\sigma$ of $i+s$ letters that shuffle $\{1,...,i\}$ with $\{i+1,...,i+s\}$
i.e. $\sigma^{-1}(1)<...<\sigma^{-1}(i)$ and $\sigma^{-1}(i+1)<...<\sigma^{-1}(i+s)$ and $(-1)^\sigma$ is the signature
of $\sigma$.

Nevertheless, $\mathcal B(G;k)$ is not necessarily a differential algebra since the Leibniz formula
\begin{eqnarray}\label{Leibniz}
\\\nonumber
\partial([x_1|...|x_i]\wedge[x_{i+1}|...|x_{i+s}])=(\partial
[x_1|...|x_i])\wedge[x_{i+1}|...|x_{i+s}]\\
+(-1)^i[x_1|...|x_i]\wedge(\partial[x_{i+1}|...|x_{i+s}])\nonumber
\end{eqnarray}
holds if and only if $x_jx_k=x_kx_j$ for all $j\le i<k$. As an immediate consequence of \eqref{Leibniz} we have the
following

\begin{lemma}\label{exteriorg}
If $x_1,...,x_i$ are elements of $G$ commuting with one another, then the element of $\mathcal B_i(G;k)$ given by
formula
$$
\langle x_1,x_2,...,x_i\rangle=[x_1]\wedge[x_2]\wedge...\wedge[x_i]
$$
is a cycle  representing a homological class in $H_i(G;k)$ which is $i$-linear and skew-symmetric in $x_1,...,x_i$.
\end{lemma}

\subsection{The Bockstein homomorphism}

If $\ell$ is a prime number and $\zeta\in G$ such that $\zeta^\ell=1$, then for each nonnegative integer $s$ we define
an element of $B_{2s}(G;k)$ given by the formula
\begin{equation*}
[\zeta]^{(s)}=\sum_{i_1,...,i_s=1}^{\ell-1}[\zeta^{i_1}|\zeta|\zeta^{i_2}|\zeta|...|\zeta^{i_s}|\zeta]
\end{equation*}
where $[\zeta]^{(0)}=[\ ]$ is the generator of $\mathcal B_0(G;k)$. By an inductive argument we can verify that
\begin{equation}\label{dividedp}
[\zeta]^{(s)}\wedge[\zeta]^{(i)}={s+i\choose i}[\zeta]^{(s+i)}
\end{equation}
for all nonnegative integers $s,i$. Again by an inductive argument using \eqref{Leibniz} and \eqref{dividedp} we can
verify the formula
\begin{equation}\label{bockstein}
\partial([\zeta]^{(s)})=\ell[\zeta]^{(s-1)}\wedge[\zeta]
\end{equation}
for all positive integers $s$. In this context, recall
\cite{MR1867354}*{p. 303} that the short exact sequence of chain
complexes
\begin{equation}\label{ses}
0\to\mathcal B_*(G;\mathbb Z/\ell)\xrightarrow{\times \ell}\mathcal B_*(G;\mathbb Z/\ell^2)\to \mathcal B_*(G;\mathbb
Z/\ell)\to 0
\end{equation}
associated with the multiplication by $\ell$ map induces a homology long exact sequence
$$
...\to H_i(G;\mathbb Z/\ell)\to H_i(G;\mathbb Z/\ell^2)\to H_i(G;\mathbb Z/\ell)\xrightarrow{\beta} H_{i-1}(G;\mathbb
Z/\ell)\to...
$$
where $\beta$ is the Bockstein homomorphism. In particular, if $\mathbb F_\ell$ denotes the field of order $\ell$ then
by a diagram chasing using \eqref{ses} and \eqref{bockstein} we obtain the following
\begin{lemma}\label{b}
If $\zeta\in G$ such that $\zeta^\ell=1$ and $s$ is a positive
integer, then $[\zeta]^{(s)}$ is a cycle representing a homology
class $\omega\in H_{2s}(G;\mathbb F_\ell)$ such that
$[\zeta]^{(s-1)}\wedge[\zeta]$ is a cycle representing the class
$\beta(\omega)\in H_{2s-1}(G;\mathbb F_\ell)$.
\end{lemma}

\subsection{The Pontryagin ring}
If $G$ is an abelian group then, according to \eqref{Leibniz},
$\mathcal B(G;k)$ is a differential graded algebra with respect to
the shuffle product \eqref{shuffle} inducing a graded algebra
structure on homology $H_*(G;k)$. If $\empty_\ell G$ denotes the
$\ell$-torsion subgroup of $G$ and $\Gamma(\empty_\ell G)$ the
algebra of divided powers \cite{Brown}*{p. 119} over $\mathbb
F_\ell$ generated in degree two by $\empty_\ell G$, then the
homomorphism of graded algebras
\begin{equation}\label{gh1}
\Gamma(\empty_\ell G)\to H_*(G;\mathbb F_\ell)
\end{equation}
sending each element $\zeta$ of $\empty_\ell G$ to the class of $[\zeta]^{(1)}$ in $H_2(G;\mathbb F_\ell)$ is well
defined according to \eqref{dividedp}. Similarly, if $\Lambda (G\otimes\mathbb Z/\ell)$ denotes the exterior algebra
over $\mathbb F_\ell$ generated in degree one by $G\otimes\mathbb Z/\ell$ then the homomorphism of graded algebras
\begin{equation}\label{eh2}
\Lambda(G\otimes\mathbb Z/\ell)\to H_*(G;\mathbb F_\ell)
\end{equation}
sending each element $g\otimes 1$ of $G\otimes\mathbb Z/\ell$ to the class of $[g]$ in $H_1(G;\mathbb F_\ell)$ is also
well defined according to Lemma \ref{exteriorg}.

\begin{proposition}[\cite{Brown} p. 126]\label{homology}
If $\ell$ is a prime number and $G$ is an abelian group, then the maps \eqref{gh1} and \eqref{eh2} induce an
isomorphism of graded algebras
$$
\Gamma(\empty_\ell G)\otimes\Lambda(G\otimes\mathbb Z/\ell)\approx H_*(G;\mathbb F_\ell).
$$
\end{proposition}

If $G_1$, $G_2$ are two groups then the K\"unneth isomorphism
\cite{MR1731415}*{p. 218}
\begin{equation}\label{k}
\kappa:H_*(G_1;\mathbb F_\ell)\otimes H_*(G_2;\mathbb F_\ell)\xrightarrow{\approx}H_*(G_1\times G_2;\mathbb F_\ell)
\end{equation}
is induced by the map sending
$$
[x_1|...|x_i]\otimes[x_{i+1}|...|x_{i+s}]\mapsto [x_1\times1|...|x_i\times1]\wedge[1\times x_{i+1}|...|1\times
x_{i+s}]$$ where $x_j$ is an element of $G_1$ for $j\le i$ and an element of $G_2$ for $j>i$. In particular, if both
$G_1$ and $G_2$ are abelian, then $\kappa$ is a graded algebra isomorphism with respect to the product
$$
(a_1\otimes b_1)(a_2\otimes b_2)=(-1)^{|b_1||a_2|}(a_1\wedge a_2)\otimes (b_1\wedge b_2)
$$
defined for homogeneous elements $a_i,b_i\in H_*(G_i;\mathbb F_\ell)$ of degrees $|a_i|$ and $|b_i|$ for $i=1,2$.

\begin{remark}\label{pontryagin} If $G$ is an abelian group and $\mu:G\times G\to G$ is its group law homomorphism, then the composition between the
induced homomorphism
$$
\mu_*:H_*(G\times G;\mathbb F_\ell)\to H_*(G;\mathbb F_\ell)
$$
and the K\"unneth isomorphism \eqref{k} for $G_1=G_2=G$ defines a product on $H_*(G;\mathbb F_\ell)$ that can be easily
checked to be induced by the shuffle product. In this case, $H_*(G;\mathbb F_\ell)$ is called the \emph{Pontryagin
ring} and its structure is given by Proposition \ref{homology}.
\end{remark}

\section{Algebras of homological symbols}\label{algebra}

Let $k$ be a fixed commutative ring with identity.

\subsection{Algebras of symbols}
If $A=\bigoplus_{i,n=0}^\infty A_{in}$ is an associative bi-graded
$k$-algebra, denote by
\begin{equation}\label{rankn}
A_{*n}:=\bigoplus_{i=0}^\infty A_{in}\subset A
\end{equation}
the $k$-submodule of all elements with the second degree $n$. Also, let
\begin{equation}\label{can}
q:T(A_{*1})=\bigoplus_{n=0}^\infty A_{*1}^{\otimes n}\to A
\end{equation}
be the canonical bi-graded algebra homomorphism where $T(A_{*1})$ is the bi-graded tensor $k$-algebra generated by the
$k$-submodule $A_{*1}\subset A$. Here $\otimes n$ denotes the $n$-fold graded tensor product over $k$.

\begin{definition}\label{s}
The \emph{algebra of symbols} associated with an associative
bi-graded $k$-algebra $A=\bigoplus_{i,n=0}^\infty A_{in}$ is the
quotient bi-graded algebra
$$
KA:=T(A_{*1})/\ker q
$$
with respect to the kernel of the canonical homomorphism \eqref{can}.
\end{definition}

\begin{definition}\label{quadratic} An associative bi-graded $k$-algebra $A=\bigoplus_{i,n=0}^\infty A_{in}$
is \emph{quadratic} with respect to the second degree if the canonical homomorphism \eqref{can} is surjective and its
kernel can be generated as a two-sided ideal by a subset of $A_{*1}^{\otimes2}$.
\end{definition}

According with the above definitions the algebra of symbols $KA$
associated with an associative bi-graded $k$-algebra $A$ comes
with a \emph{natural} bi-graded algebra monomorphism
\begin{equation}\label{mono}
q':KA\hookrightarrow A
\end{equation}
Some questions of interest will be to study when $q'$ is an isomorphism and when $KA$ is a quadratic algebra with
respect to the second degree.

\subsection{Graded $H$-spaces}
We say that a topological space $X=\bigsqcup_{n=0}^\infty X_n$ decomposed into a disjoint union of non-empty open
subspaces $X_n\subset X$ is a \emph{graded} $H$-\emph{space} if there is a continuous map $h:X\times X\to X$  with
$$h(X_n\times X_m)\subset X_{n+m}\text{ for all }n,m\ge0$$ such that $X$ is an associative $H$-space relative $h$ in the sense
of \cite{MR1867354}*{p. 281} with the homotopy unit in $X_0$.  A
continuous map between graded $H$-spaces
$$
f:X=\bigsqcup_{n=0}^\infty X_n\to Y=\bigsqcup_{n=0}^\infty Y_n
$$
is a \emph{graded} $H$-\emph{map} if $f(X_n)\subset Y_n$ for all
$n\ge0$ and $f$ is an $H$-map.
\begin{definition}\label{hs}
The \emph{$k$-algebra of homological symbols} associated with a graded $H$-space $X=\bigsqcup_{n=0}^\infty X_n$ is the
algebra of symbols $KH_*(X;k)$ associated in the sense of the Definition \ref{s} with the bi-graded $k$-algebra
$$H_*(X;k)\approx\bigoplus_{i,n=0}^\infty H_i(X_n;k)$$
where $H_*(\ \ ;k)$ is the singular homology functor with $k$-coefficients.\end{definition}

In the above definition, the bi-graded algebra structure on $H_*(X;k)$ is induced from the graded $H$-structure on $X$
via the K\"unneth homomorphisms
$$
H_*(X_n;k)\otimes H_*(X_m;k)\to H_*(X_n\times X_m;k)
$$
and the assignment $X\mapsto KH_*(X;k)$ is obviously natural with respect to graded $H$-maps. Also we have a natural
monomorphism
\begin{equation}\label{monoX}
q':KH_*(X;k)\hookrightarrow H_*(X;k).
\end{equation}
given by \eqref{mono} applied to $A=H_*(X;k)$.

\begin{notation}\label{n} For the rest of this article, if not otherwise stated, we
fix $\ell:=2r+1$ a \emph{regular} odd prime number, $\xi$ is a
primitive $\ell$-root of unity, and $R:=\mathbb
Z[\frac{1}{\ell},\xi]$ the ring of cyclotomic $\ell$-integers.
Also $GL_n$, $SL_n$ will denote the groups of matrices over $R$ as
defined in the Introduction.
\end{notation}

\section{The main examples}
In this article we are concerned with examples of algebras of homological symbols arising from linear groups.

\subsection{Approximations to $BGL_n$}
The mod $\ell$ homology of the group $GL_n$ is naturally
isomorphic to the singular mod $\ell$ homology of its classifying
space $BGL_n$. The classifying space $BGL_n$ can be approximated
by the classifying space $BGL_1^{\times n}$ of the $n$-fold direct
product $GL_1^{\times n}$ and by a topological space $BGL_n^{et}$
called the \emph{\'etale model at $\ell$}, defined in
\cite{MR1259512}*{p. 3}. These spaces are connected by natural
continuous maps
\begin{equation}\label{approx}
BGL_1^{\times n}\xrightarrow{\iota_n}BGL_n\xrightarrow{f_n}BGL_n^{\acute{e}t}
\end{equation}
where $\iota_n$ is the classifying space map induced by the
canonical inclusion $GL_1^{\times n}\subset GL_n$ and $f_n$ is a
map defined in \cite{MR1259512}*{p. 3}. By taking the disjoint
union of the diagrams \eqref{approx} we obtain a diagram of
topological spaces and continuous maps
\begin{equation}\label{hstructure}
X:=\bigsqcup_{n=0}^{\infty}BGL_1^{\times
n}\xrightarrow{\iota}Y:=\bigsqcup_{n=0}^{\infty}BGL_n\xrightarrow{f}Z:=\bigsqcup_{n=0}^{\infty}BGL_n^{\acute{e}t}
\end{equation}
such that each disjoint union has a graded $H$-space structure induced by the matrix block-multiplication and the maps
$\iota=\sqcup\iota_n$ and $f=\sqcup f_n$ are graded $H$-maps. On mod $\ell$ homology, the diagram \eqref{hstructure}
induces a commutative diagram of bi-graded algebras and homomorphisms
\begin{equation}\label{hdiagram}
\begin{CD}KH_*(X;\mathbb F_\ell)@>K\iota_*>> KH_*(Y;\mathbb F_{\ell})@>Kf_*>> KH_*(Z;\mathbb F_{\ell})\\
@VVq_1V @VVq_2V @VVq_3V\\
H_*(X;\mathbb F_\ell)@>\iota_*>> H_*(Y;\mathbb F_{\ell})@>f_*>> H_*(Z;\mathbb F_{\ell})
\end{CD}
\end{equation}
where the algebras of homological symbols in the first row are
given by the Definition \ref{hs} and the monomorphisms $q_i$ are
are given by \eqref{monoX} for $i=1,2,3$. The second row of the
diagram \eqref{hdiagram} can written as a diagram of bi-graded
algebras
\begin{equation}\label{bigrade}
T:=\bigoplus_{i,n=0}^{\infty}T_{i,n}\xrightarrow{\iota_*}A:=\bigoplus_{i,n=0}^{\infty}A_{i,n}\xrightarrow{f_*}
A^{\acute{e}t}:=\bigoplus_{i,n=0}^{\infty}A^{\acute{e}t}_{i,n}
\end{equation}
where for each bi-degree $(i,n)$, we define
\begin{equation*}
T_{i,n}:=H_i(GL_1^{\times n};\mathbb F_\ell),\ \
A_{i,n}:=H_i(GL_n;\mathbb F_\ell),\ \
A_{i,n}^{\acute{e}t}:=H_i(BGL_n^{\acute{e}t};\mathbb F_\ell).
\end{equation*}
The first degree $i$ is called the \emph{homological degree} and the second degree $n$ is called the \emph{rank}.

\begin{theorem}[\cite{MR1259512}*{Lemma 6.2}]\label{onto}
The composed homomorphism $f_*\circ\iota_*$ in the diagram
\eqref{bigrade} is surjective.
\end{theorem}

The rank $n$ elements of $T$ form the linear subspace $T_{*n}\subset T$  (see \eqref{rankn}) such that:
\begin{eqnarray*}
T_{*n}=H_*(GL_1^{\times n};\mathbb F_\ell)\approx H_*(GL_1;\mathbb F_\ell)^{\otimes n}
\end{eqnarray*}
by the K\"unneth isomorphism. In particular, $$T=H_*(X;\mathbb
F_\ell)\approx T(H_*(GL_1;\mathbb F_\ell))$$ is the tensor algebra
generated by $T_{*1}=H_*(GL_1;\mathbb F_\ell)$ and thus, $q_1$ in
\eqref{hdiagram} is an isomorphism. From Theorem \ref{onto} we
deduce the following

\begin{corollary}\label{q3}
The monomorphism $q_3$ in the diagram \eqref{hdiagram} is an isomorphism.
\end{corollary}

\subsection{\'Etale obstruction classes}\label{eoc}
To describe the kernel of $f_*\circ\iota_*$ we observe that
according to \cite{MR718674} the group of units $GL_1$ of the ring
$R$ is the abelian group generated by the set of cyclotomic units
\begin{equation}\label{cunits}
\{-\xi, 1-\xi, 1-\xi^2,...,1-\xi^r\}
\end{equation}
subject to the relation $(-\xi)^{2\ell}=1$. By applying Proposition \ref{homology} to $GL_1$, we deduce that $T_{*1}$
is a vector space over $\mathbb F_\ell$ with basis the set of homology classes represented by cycles of the form
\begin{equation}\label{cycles}
[\xi]^{(s)}\wedge\langle v_1,...,v_i\rangle
\end{equation}
where $s$ runs over all nonnegative integers and $\{v_1,...,v_i\}$
over all subsets of the set \eqref{cunits}. In this context, the
following definition is a slight modification of
\cite{MR1973992}*{p. 2336}:
\begin{definition}\label{eos}
A class $\epsilon\in T_{*1}$ represented by a cycle of the form
\eqref{cycles} is called an \emph{\'etale obstruction class} if
$s$ is a nonnegative integer and $\{v_1,...,v_i\}$ is a subset of
the set \eqref{cunits} of cardinality $i$ such that $i=s+2j$ for
some integer $j>0$.
\end{definition}

\begin{definition}\label{weight} A class $\omega\in T_{1*}$ represented by a
cycle of the form \eqref{cycles} is called a \emph{homogeneous class of weight} $\parallel\omega\parallel:=s+i$.
\end{definition}

\begin{remark}\label{e}
For each integer $i\ge 2$ let $e(i)$ denote the cardinality of the set of all integers $s\equiv i\mod2$ such that $0\le
s\le i-2$. Then the number $e$ of \'etale obstruction classes is finite and given by the formula
\begin{equation*}
e=\sum_{i=2}^{r+1}e(i){r+1\choose i}.
\end{equation*}
\end{remark}

The following group homomorphisms:
\begin{equation*}
GL_1\xrightarrow{t}GL_1^{\times2}\xleftarrow{\rho}GL_1^{\times 3}
\end{equation*}
given by the formulas
\begin{equation}\label{tro}
 t(u)=u^{-1}\times u,\ \rho(u\times v\times w)=uw\times vw,
\end{equation}
for $u,v,w\in GL_1$ induce homomorphisms on mod $\ell$ homology:
\begin{eqnarray*}
t_*:T_{*1}\to T_{*2},\ \rho_*:T_{*2}\otimes T_{*1}\approx T_{*3}\to T_{*2},
\end{eqnarray*}
where the source $T_{*3}$ of $\rho_*$ has been identified with
$T_{*2}\otimes T_{*1}$ via the K\"unneth isomorphism. With these
preparations, we have the following important result:

\begin{theorem}[\cite{MR1973992}]\label{kernel}
The kernel of the bi-graded algebra homomorphism:
$$f_*\circ\iota_*:T\to A^{\acute{e}t}$$
is the two-sided ideal of $T$ generated by the set of elements of the form:
\begin{equation}\label{eta}
\rho_*(t_*(\eta)\otimes z),
\end{equation}
where $\eta,z\in T_{*1}$ such that $\eta$ runs over all the
\'etale obstruction classes and the homogeneous classes of
\emph{odd} weight $\parallel\eta\parallel$, and $z$ runs over a
vector space basis for $T_{*1}$.
\end{theorem}

The proof of this theorem is a direct translation using Lemmas
\ref{exteriorg} and \ref{b} of the calculations made in
\cite{MR1973992}*{p. 2338}. Also, via the K\"unneth isomorphisms,
$T$ can be regarded as the tensor algebra on $T_{*1}$ and $T_{*1}$
can be identified via $f_*\circ \iota_*$ with
$A_{*1}^{\acute{e}t}$ (see \cite{MR1259512}*{Proposition 5.2}).
Thus, combining the Theorems \ref{onto} and \ref{kernel} we obtain
the structure of the bi-graded algebra $A^{\acute{e}t}$ as a
quadratic algebra with respect to the rank:

\begin{corollary}\label{qq}
The bi-graded algebra $A^{\acute{e}t}$ in \eqref{bigrade} is a
quadratic algebra with respect to the rank in the sense of the
Definition \ref{quadratic}.
\end{corollary}

\begin{remark}\label{sq}
The homomorphism $\rho_*$ defines a graded module structure on $T_{*2}$ over the Pontryagin ring $T_{*1}$ (see Remark
\ref{pontryagin}). The Theorem \ref{kernel} says that the kernel of $f_*\circ\iota_*$ is generated as a two-sided ideal
by a submodule of $T_{*2}$ of finite rank $e$ over $T_{*1}$ modulo the classes \eqref{eta} with
$\parallel\eta\parallel$ odd, where $e$ is given by Remark \ref{e}.
\end{remark}

\section{The main conjecture}\label{conj}

\subsection{The statement}\label{statement} The maps in the diagram \eqref{hdiagram} have the following known properties:

(1) $K\iota_*$ and $Kf_*$ are surjective. This is immediate from the fact that $K\iota_*$ and $Kf_*$ are bijective in
rank $1$ and their targets are generated as algebras by rank 1 elements.

(2) $f_*$ is surjective but not an isomorphism. The first part
follows from the Theorem \ref{onto} while the last part was proven
in \cite{MR1851242}.

(3) $\iota_*$ is surjective \emph{if} the Quillen conjecture
\cite{MR0298694}*{p. 591} holds true for the ring $R$ and all the
ranks $n$. This fact was proven in \cite{MR1312569}*{p. 51}.

(4) $q_1$ and $q_3$ are isomorphisms. These facts follow from the Corollary \ref{q3} and its preceding proof.

(5) $q_2$ is an isomorphism \emph{if} $\iota_*$ is surjective. This follows from (4) by chasing the diagram
\eqref{hdiagram}.

(6) $f_*$ is an isomorphism \emph{if} $Kf_*$ is bijective and $\iota_*$ is surjective. This follows from (4) and (5) by
chasing the diagram \eqref{hdiagram}.

In this article we conjecture that:

\begin{conjecture}\label{cj} The map $Kf_*:KA\to KA^{\acute{e}t}\approx A^{\acute{e}t}$ in the diagram \eqref{hdiagram} is an isomorphism.\end{conjecture}

By (2), (3) and (6), our Conjecture \ref{cj} implies that the
Quillen conjecture \cite{MR0298694}*{p. 591} for the ring $R$
defined in Notation \ref{n} cannot be true in all the ranks $n$.
In this sense, our conjecture can be regarded as a "correction" of
the Quillen conjecture. Also our conjecture implies that $\iota_*$
is not surjective and $q_2$ is not an isomorphism.

\begin{remark}\label{useful} The Conjecture \ref{cj} and the Theorems \ref{onto} and \ref{kernel} (see also Remark \ref{sq}) compute
the direct summand $KA$ of the mysterious algebra \eqref{a}. This
summand is an algebra of homological symbols which is quadratic
with respect to the rank by the Corollary \ref{qq} .
\end{remark}

\subsection{A useful reduction} Recalling $t$, $\rho$ defined in \eqref{tro}, we have a commutative diagram
\begin{equation*}
\begin{CD}
GL_1\times GL_1@>t\times Id>>GL_1^{\times 2}\times GL_1\\
@V\tau\times Id VV @VV\widetilde{\rho} V\\
SL_2\times GL_1@>\mu >> GL_2
\end{CD}
\end{equation*}
where $Id$ is the identity map, $\widetilde{\rho}$ is $\rho$ composed with the canonical inclusion $GL_1^{\times
2}\subset GL_2$,
\begin{equation}\label{tau}
\tau(u)=\begin{pmatrix}u^{-1}&0\\0&u\end{pmatrix}\text{ and } \mu(A\times u)=A\begin{pmatrix}u&0\\0&u\end{pmatrix}
\text{ (matrix product)}
\end{equation}
for all $u\in GL_1$ and $A\in SL_2$. By passing to mod $\ell$ homology we have the following

\begin{proposition}\label{reduce}
The Conjecture \ref{cj} is true if and only if $\tau_*(\epsilon)=0$ in $H_*(SL_2;\mathbb F_\ell)$ for all \'etale
obstruction classes $\epsilon\in H_*(GL_1;\mathbb F_\ell)$.
\end{proposition}

\begin{proof} The cycle $[\xi^{-1}]^{(1)}$ is homologous to $-[\xi]^{(1)}$ as we deduce from
$$
\partial[\xi^i|\xi^{-1}|\xi]=[\xi^{-1}|\xi]-[\xi^{i-1}|\xi]-[\xi^i|\xi^{-1}]
$$
by taking the sum over $i=1,...,\ell$. If $\sigma_*:T_{*1}\to
T_{*1}$ is the homomorphism induced by $\sigma:GL_1\to GL_1$,
$u\mapsto u^{-1}$, and $\eta$ is represented by \eqref{cycles}
then we can prove inductively that $
\sigma_*(\eta)=(-1)^{\parallel\eta\parallel}\eta $ where
$\parallel\eta\parallel$ is given by Definition \ref{weight}.
Because $\sigma$ extends to an inner automorphism of $SL_2$ via
$\tau$, we conclude that $\tau_*\circ\sigma_*$ is the identity map
on $H_*(SL_2;\mathbb F_\ell)$. Hence, the classes $\tau_*(\eta)$
with $\eta\in T_{*1}$ and $\parallel\eta\parallel$ odd vanish in
$H_*(SL_2;\mathbb F_\ell)$. The necessity follows now from the
equation
$$
\widetilde{\rho}_*(t_*(\eta)\otimes z)=\mu_*(\tau_*(\eta)\otimes z)
$$
by chasing the diagram \eqref{hdiagram} and using the Theorem
\ref{kernel}. The sufficiency follows by a spectral sequence
argument as in \cite{MR1973992}*{Lemma 4.8}.
\end{proof}

\section{A group theoretical approach}\label{fp}

The aim of this section is to provide a group theoretical method producing evidence for the Conjecture \ref{cj}. This
method is based on a finitely presented group defined next.

\subsection{A finitely presented group} Let $SE_2$ be the group generated by the symbols
$D(u)$ and $E(x)$ subject to the following relations
\cite{MR0207856}:
\begin{eqnarray}\label{types}
\text{ Type I. } &&E(x)E(0)E(y)=D(-1)E(x+y)\\\nonumber \text{ Type II. } &&E(x)=D(u)E(xu^2)D(u)\\\nonumber \text{
Type III. } &&E(u^{-1})E(u)E(u^{-1})=D(-u)\\
\nonumber\text{ Type IV. } &&D(u)D(v)=D(uv)
\end{eqnarray}
where $u,v\in GL_1$ and $x,y\in R$ run over all elements. We
introduce the following labels:
\begin{equation}\label{l}
z:=D(\xi),\ u_i:=D(\epsilon_i),\ a:=E(0),\ b:=E(1)
\end{equation}
where $\epsilon_i:=1-\xi^i$ for $i=1,2,...,r$ are given by
\eqref{cunits}, and we define:
\begin{equation}\label{d}
b_t:=z^{rt}bz^{rt}a,\ w:=z^cu_1u_2...u_r
\end{equation}
where $t=0,1,2,...,2r$ and $c\ge0$ is the smallest integer such
that
$$2c\equiv r^2+\frac{r(r+1)}{2}\mod\ell.$$ We will
occasionally use $b_t$ with $t$ an \emph{arbitrary} integer where
$b_t=b_s$ if $t\equiv s\mod\ell$ and the following notation
$[x,y]=xyx^{-1}y^{-1}$.

\begin{definition}\label{cc}
For each non-empty subset $I\subset\{1,2,...,r\}$ define
$$
c(I):=(\prod_{t=0}^{2r}b_t^{c_t(I)})a^{-1}\prod_{i\in I}u_i
$$
where $c_t(I)\in\mathbb Z$ such that in $R$ we have the following identity:
$$
\epsilon_I:=\prod_{i\in I}\epsilon_i=\prod_{i\in I}(1-\xi^i)=\sum_{t=0}^{2r}c_t(I)\xi^t.
$$
For instance, if $I=\{i\}$ is a singleton, then $
c(I)=b_0b_i^{-1}a^{-1}u_i$ and if $I=\{i,j\}$ has two elements
then $c(I)=b_0b_i^{-1}b_j^{-1}b_{i+j}a^{-1}u_iu_j$.
\end{definition}

\begin{theorem}\label{main} The group $SE_2$ defined above is generated by
$$
z, u_1, u_2,..., u_r, a, b
$$
subject to the following relations:
\begin{eqnarray}\label{units}
z^\ell=[z,u_i]=[u_i,u_j]=1\\\label{aa} a^4=[a^2,z]=[a^2,u_i]=1\\\label{absorb} a=zaz=u_iau_i\\\label{bb}
[b_s,b_t]=1\\\label{cyclo} b^3=a^2=b_0b_1...b_{2r}\\\label{elf} b_t^\ell=w^{-1}b_t^{(-1)^r}w\\\label{cui}
c(I)^3=1\\\label{last} ba^2=u_ibz^{-ri}b^{-1}b_0^{-1}z^{ri}bz^{-i}u_i
\end{eqnarray}
where $i,j\in\{1,2,...,r\}$, $s,t\in\{0,1,2,...,2r\}$, and $I\subset\{1,2,...,r\}$ runs over all nonempty subsets.
\end{theorem}

The theorem implies that $SE_2$ has a finite presentation with $r+3$ generators and $6+6.5r+2.5r^2+2^r$ relators. Its
proof will be given as a sequence of lemmas. For convenience, we will refer to the relations \eqref{types} only by
type. Also we will tacitly use \eqref{l}, \eqref{d}, the relations in $GL_1$ given at the beginning of \S \ref{eoc},
and when appropriately, Type IV.

\begin{lemma}
$z,u_1,u_2,...,u_r,a,b$ generate $SE_2$.
\end{lemma}

\begin{proof}
 By Type II with $u=-1$, it follows that $D(-1)$ is central and by Type I with $x=y=0$, we have
\begin{equation}\label{aad} a^2=D(-1).\end{equation}
Since each $v\in GL_1$ can be written as
$v=(-\xi)^j\epsilon_1^{a_1}...\epsilon_r^{a_r}$ for some integers $j$, $a_1$,...,$a_r$, we have
\begin{equation}\label{dv}D(v)=a^{2j}z^ju_1^{a_1}...u_r^{a_r}.\end{equation}
By Type II with $u=\xi^{rt}$ and $x=\xi^{-2rt}=\xi^t$,
\begin{equation}\label{bt}b_t=E(\xi^{t})E(0).\end{equation}
By Type I with $y=-x$ and \eqref{aad},
$$
E(x)E(0)E(-x)E(0)=1
$$
and hence,
\begin{equation}\label{bti}
b_t^{-1}=E(-\xi^t)E(0).
\end{equation}
If $x'=\sum_{t=0}^{2r}m_t\xi^t$ in $R$ with $m_t$ integers, then, by Type I,
$$
E(x')=[\prod_{t=0}^{2r}(E(\xi^t)E(0))^{m_t^+}(E(-\xi^t)E(0))^{m_t^-}]E(0)^{-1}D(-1)^{m-1}
$$
where $m=\sum_{t=0}^{2r}m_t$, and $m_t=m_t^+-m_t^-$ with $m_t^+$, $m_t^-$ nonnegative integers. Combining \eqref{aad},
\eqref{bt}, \eqref{bti} with the equation above, we deduce that
\begin{equation}\label{ey}
E(x')=(\prod_{t=0}^{2r}b_t^{m_t})a^{2m-3}.
\end{equation}
We remark that a permutation of the $\xi^t$-terms in $x'$ corresponds to a permutation of the $b_t$-factors in $E(x')$.
Any ring element $x\in R$ can be written in the form $x=x'v^{-2}$ for some $x'\in\mathbb Z[\xi]$ and $v\in GL_1$. By
Type II, we have
\begin{equation}\label{final}
E(x)=D(v)E(x')D(v)
\end{equation}
with $D(v)$ given by \eqref{dv} and $E(x')$ by \eqref{ey}, concluding the proof.

\end{proof}

\begin{lemma}\label{necessary1}
The relations \eqref{units} - \eqref{cyclo} are \emph{necessary}.
\end{lemma}

\begin{proof} We have the following list of short arguments:

\eqref{units} follows from Type IV.

\eqref{aa} follows from \eqref{aad}.

\eqref{absorb} follows from Type II with $x=0$.

\eqref{bb} follows from \eqref{ey} with $x'=\xi^t+\xi^s=\xi^s+\xi^t$ in $R$.

\eqref{cyclo} \emph{the first part} follows from Type III with
$u=1$ and \eqref{aad}.

\eqref{cyclo} \emph{the second part} follows by \eqref{ey} with
$x'=\sum_{j=0}^{2r}\xi^j=0$ in $R$.

\end{proof}

\begin{lemma}[\cite{MR718674}]\label{factor}
In $R$ we have $\ell=(-1)^r\lambda^2$ where
$\lambda:=\xi^c\epsilon_1\epsilon_2...\epsilon_r$.
\end{lemma}

\begin{lemma} \eqref{elf} is \emph{necessary}.
\end{lemma}

\begin{proof} By Lemma \ref{factor} we can apply \eqref{final} to $$x=\ell\xi^t, x'=(-1)^r\xi^t, v=\lambda^{-1}$$
and get
$$
E(\ell\xi^t)=D(\lambda)^{-1}E((-1)^r\xi^t)D(\lambda)^{-1}.
$$
By \eqref{dv} and \eqref{ey}, the equation above can be rewritten as
$$
b_t^\ell a^{4r-1}=w^{-1}b_t^{(-1)^r}a^{-1}w^{-1}.
$$
Now we can use \eqref{aa} and \eqref{absorb} proven in Lemma \ref{necessary1}.
\end{proof}

\begin{lemma}\label{pcui} \eqref{cui} is \emph{necessary}.
\end{lemma}

\begin{proof}

\noindent For $I\subset\{1,2,...,r\}$ recall that $\epsilon_I:=\prod_{i\in I}\epsilon_i$. Then, the Definition \ref{cc}
gives by \eqref{ey} with $x'=\epsilon_I$ the following formula
$$
c(I)=D(-1)E(\epsilon_I)D(\epsilon_I).
$$
By \eqref{final} with $x=\epsilon_I^{-1}$ and $x'=v=\epsilon_I$ we have
$$
E(\epsilon_I^{-1})=D(\epsilon_I)E(\epsilon_I)D(\epsilon_I).
$$
By Type III with $u=\epsilon_I$, we have
$$
D(-1)E(\epsilon_I)E(\epsilon_I^{-1})E(\epsilon_I)D(\epsilon_I)=1.
$$
The conclusion follows by combining the three equations above.
\end{proof}

\begin{lemma}\label{necessary2} \eqref{last} is \emph{necessary}.
\end{lemma}

\begin{proof}
We start with $\epsilon_i^2=\xi^{i}(\xi^{-i}-2+\xi^i)$ in $R$ and by Type II with $u=\epsilon_i^{-1}$, $x=\epsilon_i^2$
and \eqref{final} with $$x'=\xi^{-i}-2+\xi^i, v=\xi^{ri}, x=\xi^ix'$$ we get
$$
u_i^{-1}bu_i^{-1}=z^{ri}b_{-i}b_0^{-1}b_0^{-1}b_ia^{-3}z^{ri}
$$
The desired relation now follows by \eqref{d} and \eqref{aa}.
\end{proof}

\begin{lemma}\label{dvS}
The relations \eqref{units} - \eqref{last} are \emph{sufficient}
to verify that 1) the relation \eqref{dv} is well defined for
$v\in GL_1$, 2) Type IV holds true, and 3) $a^2=D(-1)$ is central.
\end{lemma}

The proof is immediate by \eqref{units}, \eqref{aa}, and
\eqref{cyclo} \emph{the first part}. In what follows we will use
this lemma tacitly.

\begin{lemma}\label{eyS}
The relations \eqref{units} - \eqref{last} are \emph{sufficient} to verify that \eqref{ey} is well defined for
$x'\in\mathbb Z[\xi]$.
\end{lemma}

\begin{proof}
Let $x'=\sum_{t=0}^{2r}m_t\xi^t=\sum_{t=0}^{2r}n_t\xi^t$ in $R$
with $m_t,\ n_t\in\mathbb Z$. Then $m_t-n_t=j$ is independent of
$t$. From \eqref{bb} and \eqref{cyclo} \emph{the second part} we
deduce that the right hand side of \eqref{ey} remains unchanged
under the transformation $m_t=n_t+j$ or a permutation of the
$b_t$-factors.
\end{proof}

\begin{lemma}\label{eyI}
The relations \eqref{units} - \eqref{last} are \emph{sufficient} for Type I with $x,y\in\mathbb Z[\xi]$.
\end{lemma}

\begin{proof}
Let $x=\sum_{t=0}^{2r}m_t\xi^t$ and $y=\sum_{t=0}^{2r}n_t\xi^t$ with $m_t$, $n_t$ integers. By Lemma \ref{eyS} we can
choose $x+y=\sum_{t=0}^{2r}(m_t+n_t)\xi^t$ and Type I follows from \eqref{ey} and \eqref{bb}.
\end{proof}

\begin{lemma}\label{IIS}
The relations \eqref{units} - \eqref{last} are \emph{sufficient} to verify that \eqref{final} is well defined for
$x=x'v^{-2}$ with $x'\in\mathbb Z[\xi]$ and $v\in GL_1$.
\end{lemma}

\begin{proof} It suffices to prove that the following statement\\

$P(x',v)$ : \emph{If $y':=x'v^{-2}\in\mathbb Z[\xi]$ then $D(v)E(x')D(v)=E(y')$ is a consequence of the relations
\eqref{units} - \eqref{last}.}\\

\noindent is true for all $x'\in\mathbb Z[\xi]$ and $v\in GL_1$ where $E(x')$, $E(y')$, and $D(v)$ are given by
\eqref{ey} and \eqref{dv}. By Lemmas \ref{eyS} and \ref{dvS}, these formulas are independent of the way $x'$, $y'$, and
$v$ are presented. Also, we recall that $b_t=b_s$ if $t\equiv s\mod\ell$.

$P(\pm\xi^t,-\xi)$ \emph{is true}. If $x'=\xi^t$ and $y'=\xi^{t-2}$, we check that
$$
zb_ta^{-1}z=b_{t-2}a^{-1}
$$
holds true by definitions. The case $x'=-\xi^t$ is similar.

$P(\pm\xi^t,\epsilon_i^{-1})$ \emph{is true}. If $x'=\xi^t$ and $y'=\xi^t-2\xi^{t+i}+\xi^{t+2i}$, we use \eqref{absorb}
to reduce the equation
$$
u_i^{-1}b_ta^{-1}u_i^{-1}=b_sb_{t+i}^{-2}b_{t+2i}a^{-3}
$$
to \eqref{last} as in the proof of Lemma \ref{necessary2}. The case $x'=-\xi^t$ is similar.

$P(\pm\ell\xi^t,\lambda)$ \emph{is true}. Here $\lambda$ is defined in Lemma \ref{factor} such that
$$
y'=x'\ell\lambda^{-2}=\pm(-1)^r\xi^t
$$
The statement now follows from \eqref{absorb} and \eqref{elf}.

By \eqref{absorb} and Lemma \ref{eyI}, $P(x_1',v)$ and $P(x_2',v)$ imply $P(x_1'+x_2',v)$.  So, $P(x',-\xi)$,
$P(x',\epsilon_i^{-1})$, and $P(\ell x',\lambda)$ are true for all $x'\in\mathbb Z[\xi]$ and $i=1,2,...,r$. If
$v_1^{-1},v_2^{-1}\in\mathbb Z[\xi]$ such that $P(x',v_1)$ and $P(x',v_2)$ are true for all $x'\in\mathbb Z[\xi]$, then
$P(x',v_1v_2)$ is also true for all $x'\in\mathbb Z[\xi]$. Since \eqref{cunits} is a generating set for $GL_1$ it
follows that $P(x',v)$ is true for all $x'\in\mathbb Z[\xi]$ and all $v\in GL_1$ such that $v^{-1}\in\mathbb Z[\xi]$.
The proof can now be concluded by the observation that every element of $GL_1$ is of the form $v\lambda^s$ with
$v^{-1}\in\mathbb Z[\xi]$ and $s$ a nonnegative integer.
\end{proof}

\begin{lemma}
The relations \eqref{units} - \eqref{last} are \emph{sufficient} for Type I, Type II, and Type III.
\end{lemma}

\begin{proof}
Type I: Given two ring elements $x,y\in R$ there exists $v\in GL_1$ such that $x=x'v^{-2}$ and $y=y'v^{-2}$ with
$x',y'\in\mathbb Z[\xi]$. By \eqref{absorb} and Lemma \ref{IIS} we get
$$
E(x)E(0)E(y)=D(v)E(x')E(0)E(y')D(v).
$$
So Type I is reduced to Lemma \ref{eyI}.

Type II follows from  Lemma \ref{IIS}.

Type III: By Lemmas \ref{eyI} and \ref{IIS} we can reverse the
proof of Lemma \ref{pcui} to conclude that Type III with
$u=\epsilon_I=\prod_{i\in I}\epsilon_i$ follows from \eqref{cui}
for $I\subset\{1,2,...,r\}$ non-empty and from \eqref{cyclo}
\emph{the first part} if $I$ is empty i.e. $u=1$. Combining this
with Type II, we deduce that Type III holds with $u=\epsilon_Iv^2$
for any $v\in GL_1$ and any subset $I$. Moreover, the Type I
implies
$$
E(-u)=E(0)^{-1}E(u)^{-1}E(0)^{-1}
$$
and hence, if Type III holds for $u\in GL_1$ then it holds for $-u$ as well. Since $\pm\epsilon_I$'s form a set of
coset representatives for $GL_1$ modulo the squares, Type III holds in general.
\end{proof}

\section{Hopf's formula calculations}\label{hopff}

There is a group homomorphism $\pi:SE_2\to SL_2$ given by
$$
D(u)\mapsto\begin{pmatrix}u^{-1}&0\\0&u\end{pmatrix},\ E(x)\mapsto \begin{pmatrix}x&1\\-1&0\end{pmatrix}
$$
for all $u\in GL_1$ and $x\in R$. Regarding $D:GL_1\to SE_2$ as a group homomorphism, we have the following commutative
diagram

\begin{equation}\label{freduction}
\begin{CD}
H_p(GL_1;\mathbb F_\ell)@>D_*>> H_p(SE_2;\mathbb F_\ell)\\
@VV\tau_*V @|\\
H_p(SL_2;\mathbb F_\ell)@<\pi_*<< H_p(SE_2;\mathbb F_\ell)
\end{CD}
\end{equation}
where $p$ is a positive integer and $\tau_*$ is induced by \eqref{tau}. Chasing this diagram, by Proposition
\ref{reduce} and Definition \ref{eos} we deduce that
\begin{proposition}\label{finite}
The Conjecture \ref{cj} is true if for each subset $\{e_1,...,e_i\}$ of $\{z,u_1,...,u_r\}$ with $2\le i\le r+1$
elements and for each pair $(s,j)$ of nonnegative integers with $i=s+2j$ and $j>0$, the standard cycle
\begin{equation}\label{ec}
[z]^{(s)}\wedge\langle e_1,...,e_i\rangle
\end{equation}
represents the zero class in $H_p(SE_2;\mathbb F_\ell)$ where $z$, $u_1$,..., $u_r$ are elements of $SE_2$ defined by
\eqref{l} and $p=3s+2j$.
\end{proposition}

According to this proposition, for each prime $\ell=2r+1$, Conjecture \ref{cj} follows from a verification that a
certain \emph{finite} set of explicitly given cycles \eqref{ec} represent the zero class in $H_*(SE_2;\mathbb F_\ell)$.
In particular, this set of cycles in $H_2(SE_2;\mathbb F_\ell)$ is given by $\langle e_1,e_2\rangle$ for $e_1, e_2$ in
$\{z,u_1,...,u_r\}$. Theorem \ref{main} gives a short exact sequence
$$
1\to K\to F\to SE_2\to 1
$$
where $F$ is the free group generated by $z$, $u_i$, $a$, $b$,
$b_t$, and $w$ for $1\le i\le r$ and $0\le t\le 2r$, and $K\subset
F$ is the normal subgroup given by the relators associated with
the relations \eqref{d} and \eqref{units} - \eqref{last}.
Associated with this free presentation, Hopf's formula
\cite{Brown}*{p. 42} identifies
\begin{equation}\label{Hopf}
H_2(SE_2;\mathbb Z)\approx \frac{K\cap[F,F]}{[F,K]}
\end{equation}
such that the standard cycle $\langle e_1,e_2\rangle$ with integer coefficients corresponds to the commutator
$[e_1,e_2]\mod [F,K]$. Here $[X,Y]$ denotes the group generated by the commutators $[x,y]$ with $x\in X$ and $y\in Y$.

\begin{lemma}\label{perfect}
$SE_2$ is a perfect group.
\end{lemma}

\begin{proof}
For $x,y\in SE_2$, let $x\equiv y$ mean that $xy^{-1}$ is a product of commutators in $SE_2$. By \eqref{units} and
\eqref{absorb} we deduce that $z^\ell\equiv z^2\equiv 1$ and hence, $z\equiv 1$ since $\ell$ is odd. Now \eqref{d}
implies $b_t\equiv ba$ for all $t$. Combining this with \eqref{absorb} and \eqref{last}, we get  $ba^3\equiv
u_i^2\equiv 1$. Since $a^4\equiv 1$ by \eqref{aa}, we conclude that $b\equiv a$, and since $b^3\equiv a^2$ by
\eqref{cyclo} we conclude that $b\equiv a\equiv 1$. Finally, from \eqref{cui} with $I=\{i\}$ \emph{singleton} (see
Definition \ref{cc}) we get $u_i^3\equiv 1$ and since $u_i^2\equiv 1$ we deduce that $u_i\equiv1$ for all
$i=1,2,...,r$. Thus, all generators of $SE_2$ are $\equiv 1$.
\end{proof}

By Lemma \ref{perfect} and the universal coefficients, from \eqref{Hopf} we have
\begin{equation}\label{elHopf}
H_2(SE_2;\mathbb F_\ell)\approx \frac{(K\cap[F,F])K^\ell}{[F,K]K^\ell}
\end{equation}
where $K^\ell$ is the normal subgroup whose relators are the $\ell$-th powers of the relators of $K$. With these
preparations, the following result is evidence for the Conjecture \ref{cj}:

\begin{proposition}\label{evidence}
If $\ell\in\{3,5\}$, then $[e_1,e_2]\in [F,K]K^\ell$ for all $e_1,e_2$ in $\{z,u_1,...,u_r\}$.
\end{proposition}

\noindent The proof of this proposition is given next based on GAP
\cite{gap}. We remark that the case $\ell=3$ is known
\cite{MR1723188} but the proof given here and the case $\ell=5$
are new.
\\

\textbf{The Case $\ell=3$.} The free group $F$ is given by
\begin{verbatim}
F:=FreeGroup(8)
z:=F.1; u1:=F.2; a:=F.3; b:=F.4;
b0:=F.5; b1:=F.6; b2:=F.7; w:=F.8;
\end{verbatim}

\noindent The relators of $K$ are given in Theorem \ref{main} for $\ell=3$ by the list

\begin{verbatim}
k:=[b0^-1*b*a,
b1^-1*z*b*z*a,
b2^-1*z^2*b*z^2*a,
w^-1*z*u1,
z^3, z*u1*z^-1*u1^-1,
a^4, a^2*z*a^-2*z^-1,
a^2*u1*a^-2*u1^-1,
z*a*z*a^-1,
u1*a*u1*a^-1,
b0*b1*b0^-1*b1^-1,
b0*b2*b0^-1*b2^-1,
b1*b2*b1^-1*b2^-1,
b^3*a^-2, b0*b1*b2*a^-2,
b0^-3*w^-1*b0^-1*w,
b1^-3*w^-1*b1^-1*w,
b2^-3*w^-1*b2^-1*w,
(b0*b1^-1*a^-1*u1)^3,
a^2*b^-1*u1*b*z^2*b^-1*b0^-1*z*b*z^2*u1];
\end{verbatim}

\noindent The relators for $K^3$ are given by the list

\begin{verbatim}
k3:=List(k,x->x^3);
\end{verbatim}

\noindent The relators for $[F,K]$ are given by the following algorithm

\begin{verbatim}
c:=function(i,j) return Comm(i,j);end;;
f:=GeneratorsOfGroup(F);
fk:=ListX(f,k,c);
\end{verbatim}

\noindent The only commutator of the form $[e_1,e_2]$ in
Proposition \ref{evidence} for $\ell=3$ is the word "k[6]", i.e.
the sixth on the list "k". To check that "k[6]" belongs to
$[F,K]K^3$ we use the following algorithm:

 \begin{verbatim}
H:=F/Concatenation(fk,k3);
RequirePackage("kbmag");
RH:=KBMAGRewritingSystem(H);
OR:=OptionsRecordOfKBMAGRewritingSystem(RH);
OR.maxeqns:=500000;
OR.tidyint:=1000;
OR.confnum:=100;
MakeConfluent(RH);
ReducedWord(RH,k[6]);
<identity...>
 \end{verbatim}

\textbf{The Case $\ell=5$.} The free group $F$ is given by
\begin{verbatim}
F:=FreeGroup(11);
z:=F.1; u1:=F.2; u2:=F.3; a:=F.4; b:=F.5;
b0:=F.6; b1:=F.7; b2:=F.8; b3:=F.9; b4:=F.10; w:=F.11;
\end{verbatim}

\noindent The relators of $K$ are given in Theorem \ref{main} for $\ell=5$ by the list

\begin{verbatim}
k:=[b0^-1*b*a,
b1^-1*z^2*b*z^2*a,
b2^-1*z^4*b*z^4*a,
b3^-1*z*b*z*a,
b4^-1*z^3*b*z^3*a,
w^-1*z*u1*u2,
z^5, z*u1*z^-1*u1^-1,
z*u2*z^-1*u2^-1,
u1*u2*u1^-1*u2^-1,
a^4, a^2*z*a^-2*z^-1,
a^2*u1*a^-2*u1^-1,
a^2*u2*a^-2*u2^-1,
z*a*z*a^-1,
u1*a*u1*a^-1,
u2*a*u2*a^-1,
b0*b1*b0^-1*b1^-1,
b0*b2*b0^-1*b2^-1,
b0*b3*b0^-1*b3^-1,
b0*b4*b0^-1*b4^-1,
b1*b2*b1^-1*b2^-1,
b1*b3*b1^-1*b3^-1,
b1*b4*b1^-1*b4^-1,
b2*b3*b2^-1*b3^-1,
b2*b4*b2^-1*b4^-1,
b3*b4*b3^-1*b4^-1,
b^3*a^-2, b0*b1*b2*b3*b4*a^-2,
b0^-5*w^-1*b0*w,
b1^-5*w^-1*b1*w,
b2^-5*w^-1*b2*w,
b3^-5*w^-1*b3*w,
b4^-5*w^-1*b4*w,
(b0*b1^-1*a^-1*u1)^3,
(b0*b2^-1*a^-1*u2)^3,
(b0*b1^-1*b2^-1*b3*a^-1*u1*u2)^3,
a^2*b^-1*u1*b*z^3*b^-1*b0^-1*z^2*b*z^4*u1,
a^2*b^-1*u2*b*z*b^-1*b0^-1*z^4*b*z^3*u2];
\end{verbatim}

\noindent The relators of $K^5$ are given by the list

\begin{verbatim}
k5:=List(k,x->x^5);
\end{verbatim}

\noindent The relators of $[F,K]$ are given by a list "fk" via the
same algorithm as for the case $\ell=3$ but applied to the new "f"
and "k". The only commutators of the form $[e_1,e_2]$ in
Proposition \ref{evidence} are the words "k[8]", "k[9]", and
"k[10]" but the algorithm used in the case $\ell=3$ is
inconclusive in the case $\ell=5$ due to its increased complexity.
For this reason, we show that these words belong to $[F,K]K^5$ by
proving the following

\begin{lemma}
$[F,F]\cap K\subset [F,K]K^5$.
\end{lemma}

\begin{proof} By trial and error we find a sublist "e$\subset$ k" of $11$ elements

\begin{verbatim}
e:=k{[5,6,15,16,17,30,31,32,33,34,37]};
\end{verbatim}

\noindent such that the complementary sublist "n$\subset$ k"

\begin{verbatim}
n:=k{[1,2,3,4,7,8,9,10,11,12,13,14,18,19,20,
21,22,23,24,25,26,27,28,29,35,36,39,38]};
\end{verbatim}

\noindent consists of elements vanishing "mod e" i.e. represent zero in the group

\begin{verbatim}
t:=F/Concatenation(fk,k5,e);
\end{verbatim}

\noindent according to the following algorithm:

\begin{verbatim}
RequirePackage("kbmag");
Rt:=KBMAGRewritingSystem(t);;
OR:=OptionsRecordOfKBMAGRewritingSystem(Rt);
OR.maxeqns:=500000;
OR.tidyint:=1000;
OR.confnum:=100;
MakeConfluent(Rt);
nt:=List([1..Length(n)],i->ReducedWord(Rt,n[i]));
<identity...>
\end{verbatim}

\noindent This means that the group $K/[F,K]K^5$ is generated by the elements in "e". The commutator group $[F,F]$ is
given by the list of relators
\begin{verbatim}
ff:=ListX(f,f,\<,c);
\end{verbatim}
\noindent The "reduced" group $F/[F,F]K^5$ is given by
\begin{verbatim}
h:=F/Concatenation(ff,k5);
\end{verbatim}
\noindent Observe that "h" is a vector space of dimension $11$
over $\mathbb F_5$ by using
\begin{verbatim}
typeh:=AbelianInvariants(h);
\end{verbatim}
\noindent Moreover the elements in the list "e" generate "h" since $s=1$ where
\begin{verbatim}
s:=Size(F/Concatenation(ff,k5,e));
\end{verbatim}
Putting these facts together and using formula \eqref{elHopf} we conclude that there is a short exact sequence
$$
0\to H_2(SE_2;\mathbb F_5)\to \frac{K}{[F,K]K^5}\to\frac{F}{[F,F]K^5}\to0
$$
where the last term is a vector space of dimension $11$ while the middle term is a vector space of dimension at most
$11$ being generated by the elements in the list "e". So that $H_2(SE_2;\mathbb F_5)=0$.
\end{proof}

By \cite{MR0207856}*{p. 7}, the canonical homomorphism
$\pi:SE_2\to SL_2$ is a group isomorphism if the ring $R$ is
Euclidean and by \cite{MR0387257} the ring $R$ is indeed Euclidean
for $\ell=5$. Hence, we deduce the following
\begin{corollary}\label{proof}
$H_2(SL_2;\mathbb F_5)=0$.
\end{corollary}


\begin{bibdiv}
\begin{biblist}

\bib{MR1723188}{article}{
   author={Anton, Marian F.},
   title={On a conjecture of Quillen at the prime $3$},
   journal={J. Pure Appl. Algebra},
   volume={144},
   date={1999},
   number={1},
   pages={1--20},
   issn={0022-4049},
   review={\MR{1723188 (2000m:19003)}},
}

\bib{MR1851242}{article}{
   author={Anton, Marian Florin},
   title={Etale approximations and the mod $l$ cohomology of ${\rm GL}\sb n$},
   booktitle={Cohomological methods in homotopy theory},
      address={Bellaterra},
      date={1998},
      series={Progr. Math.},
      volume={196},
      publisher={Birkh\"auser},
      place={Basel},
   date={2001},
   pages={1--10},
   review={\MR{1851242 (2002h:11119)}},
}

\bib{MR1973992}{article}{
   author={Anton, Marian F.},
   title={An elementary invariant problem and general linear group
   cohomology restricted to the diagonal subgroup},
   journal={Trans. Amer. Math. Soc.},
   volume={355},
   date={2003},
   number={6},
   pages={2327--2340 (electronic)},
   issn={0002-9947},
   review={\MR{1973992 (2004c:57061)}},
}

\bib{Brown}{book}{
   author={Brown, Kenneth S.},
   title={Cohomology of groups},
   series={Graduate Texts in Mathematics},
   volume={87},
   publisher={Springer-Verlag},
   place={New York},
   date={1994},
   pages={x+306},
   isbn={0-387-90688-6},
   review={\MR{1324339 (96a:20072)}},
}

\bib{MR1731415}{book}{
   author={Cartan, Henri},
   author={Eilenberg, Samuel},
   title={Homological algebra},
   series={Princeton Landmarks in Mathematics},
   note={With an appendix by David A. Buchsbaum;
   Reprint of the 1956 original},
   publisher={Princeton University Press},
   place={Princeton, NJ},
   date={1999},
   pages={xvi+390},
   isbn={0-691-04991-2},
   review={\MR{1731415 (2000h:18022)}},
}


\bib{MR0207856}{article}{
   author={Cohn, P. M.},
   title={On the structure of the ${\rm GL}\sb{2}$ of a ring},
   journal={Inst. Hautes \'Etudes Sci. Publ. Math.},
   number={30},
   date={1966},
   pages={5--53},
   issn={0073-8301},
   review={\MR{0207856 (34 \#7670)}},
}


\bib{MR1443381}{article}{
   author={Dwyer, W. G.},
   title={Exotic cohomology for ${\rm GL}\sb n({\bf Z}[1/2])$},
   journal={Proc. Amer. Math. Soc.},
   volume={126},
   date={1998},
   number={7},
   pages={2159--2167},
   issn={0002-9939},
   review={\MR{1443381 (2000a:57092)}},
}



\bib{MR1259512}{article}{
   author={Dwyer, William G.},
   author={Friedlander, Eric M.},
   title={Topological models for arithmetic},
   journal={Topology},
   volume={33},
   date={1994},
   number={1},
   pages={1--24},
   issn={0040-9383},
   review={\MR{1259512 (95h:19004)}},
}



\bib{gap}{book}{
   author={The GAP~Group},
   title={GAP -- Groups, Algorithms, and Programming},
   series={Version 4.4.10},
   note={http://www.gap-system.org},
   date={2007}}

\bib{MR1867354}{book}{
   author={Hatcher, Allen},
   title={Algebraic topology},
   publisher={Cambridge University Press},
   place={Cambridge},
   date={2002},
   pages={xii+544},
   isbn={0-521-79160-X},
   isbn={0-521-79540-0},
   review={\MR{1867354 (2002k:55001)}},
}



\bib{MR1683179}{article}{
   author={Henn, Hans-Werner},
   title={The cohomology of ${\rm SL}(3,{\bf Z}[1/2])$},
   journal={$K$-Theory},
   volume={16},
   date={1999},
   number={4},
   pages={299--359},
   issn={0920-3036},
   review={\MR{1683179 (2000g:20087)}},
}

\bib{MR1312569}{article}{
   author={Henn, Hans-Werner},
   author={Lannes, Jean},
   author={Schwartz, Lionel},
   title={Localizations of unstable $A$-modules and equivariant mod $p$
   cohomology},
   journal={Math. Ann.},
   volume={301},
   date={1995},
   number={1},
   pages={23--68},
   issn={0025-5831},
   review={\MR{1312569 (95k:55036)}},
}

\bib{MR0387257}{article}{
   author={Lenstra, H. W., Jr.},
   title={Euclid's algorithm in cyclotomic fields},
   journal={J. London Math. Soc. (2)},
   volume={10},
   date={1975},
   number={4},
   pages={457--465},
   issn={0024-6107},
   review={\MR{0387257 (52 \#8100)}},
}



\bib{MR0260844}{article}{
   author={Milnor, John},
   title={Algebraic $K$-theory and quadratic forms},
   journal={Invent. Math.},
   volume={9},
   date={1969/1970},
   pages={318--344},
   issn={0020-9910},
   review={\MR{0260844 (41 \#5465)}},
}


\bib{MR1147814}{article}{
   author={Mitchell, Stephen A.},
   title={On the plus construction for $B{\rm GL}\,{\bf Z}[\frac12]$ at the
   prime $2$},
   journal={Math. Z.},
   volume={209},
   date={1992},
   number={2},
   pages={205--222},
   issn={0025-5874},
   review={\MR{1147814 (93b:55021)}},
}


\bib{MR0298694}{article}{
   author={Quillen, Daniel},
   title={The spectrum of an equivariant cohomology ring. I, II},
   journal={Ann. of Math. (2)},
   volume={94},
   date={1971},
   pages={549--572; ibid. (2) 94 (1971), 573--602},
   issn={0003-486X},
   review={\MR{0298694 (45 \#7743)}},
}


\bib{MR997360}{article}{
   author={St{\"o}hr, Ralph},
   title={A generalized Hopf formula for higher homology groups},
   journal={Comment. Math. Helv.},
   volume={64},
   date={1989},
   number={2},
   pages={187--199},
   issn={0010-2571},
   review={\MR{997360 (90d:20091)}},
}

\bib{MR718674}{book}{
   author={Washington, Lawrence C.},
   title={Introduction to cyclotomic fields},
   series={Graduate Texts in Mathematics},
   volume={83},
   publisher={Springer-Verlag},
   place={New York},
   date={1982},
   pages={xi+389},
   isbn={0-387-90622-3},
   review={\MR{718674 (85g:11001)}},
}


\end{biblist}
\end{bibdiv}
\end{document}